\documentclass{article}
\usepackage{amsmath}
\usepackage{amssymb}
\usepackage{amsthm}

 \newtheorem{thm}{Theorem}[section]
 \newtheorem{cor}[thm]{Corollary}
 \newtheorem{lem}[thm]{Lemma}
 \newtheorem{prop}[thm]{Proposition}
 \theoremstyle{definition}
 \newtheorem{defn}[thm]{Definition}
 \theoremstyle{remark}
 \newtheorem{rem}[thm]{Remark}
 
 \numberwithin{equation}{section}
 
\renewcommand{\vec}[1]{\mathbf{#1}}

\DeclareMathOperator*{\mspan}{span}
\newcommand{\comp}{\alpha} 

\newcommand{\vecsp}[1]{#1} 

\renewcommand{\div}{\mathrm{div}\,} 
\newcommand{\uinc}{\boldsymbol v} 
\newcommand{\uinco}{\uinc_0} 
\newcommand{\pinc}{q} 
\newcommand{\rinc}{\rho_0} 
\newcommand{\db}[1]{\left\langle#1\right\rangle} 
\newcommand{\setof}[2]{\left\{\left.#1\vphantom{#2} \right| \, #2 \right\}}
\newcommand{\wto}{\rightharpoonup} 
\newcommand{\swto}{\rightharpoondown} 

\begin{document}

%
%
%
%
%
%
%
%

\title{Asymptotic Properties of Linearized Equations of 
Low Compressible Fluid Motion\footnote{%
\textit{Submitted to Journal of Mathematical Fluid Mechanics.}
}}

\author{N.A. Gusev\footnote{%
The author has been supported by
 FTP ``Scientific and scientific--pedagogical personnel of innovative Russia''
 (state contracts P532 and P938) and RFBR Grant 09-01-12157-ofi-m.
}}





\date{December 21, 2010}

\maketitle
\begin{abstract}
Initial--boundary value problem for linearized equations of motion of
viscous barotropic fluid in a bounded domain is considered.
Existence, uniqueness and estimates of weak solutions to 
this problem are derived. Convergence of the solutions
towards the incompressible limit when compressibility
tends to zero is studied.
\end{abstract}

\section{Introduction}
In many cases mathematical treatment of liquids is done
in the framework of \emph{incompressible} fluid. However,
from the physical point of view, all the liquids existing
in nature are \emph{low compressible}. Therefore
it is reasonable to study asymptotic properties of solutions
to equations of low compressible fluid, in particular,
convergence to the corresponding incompressible limit.

Low Mach number limit, which
can be considered as a particular case of low compressibility limit,
was studied by
E. Feireisl, P.-L. Lions and others 
\cite{FeireislNovotnyNS,FeireislBook,FeireislNovotnyOB,LionsMasmoudi}.
In particular, it was
proved (see \cite{FeireislNovotnyNS,LionsMasmoudi})
that there exists a sequence of weak solutions
to equations of compressible fluid motion
which converges to a solution
of the corresponding equations of incompressible fluid.
More precisely, \emph{weak} convergence of the velocity was
established. But from the physical consideration, one
could desire \emph{strong} convergence of the solutions
to yield better approximation of compressible fluid
by incompressible one. Therefore it is interesting to
study sufficient conditions for the convergence to be strong.

Strong convergence of the velocity was established in
\cite[III, \S 8]{Temam}
for the solutions of ``compressible'' system arising
in the method of artificial compressibility. It was also
proved that the gradient of the pressure converges weakly,
but strong convergence of the pressure was not examined.

In some situations the convergence
to the incompressible limit cannot be strong. In \cite{Shifrin}
there was derived a \emph{necessary} condition of
strong convergence of \emph{classical} solutions to
equations of compressible fluid when compressibility tends to zero.
The condition represents a restriction on the initial condition
for the equations of incompressible fluid and this restriction
cannot be satisfied in general case.

In this paper we study \emph{weak} solutions to 
\emph{linearized} equations of compressible fluid.
The reason for this is twofold. First, by the present moment
only existence of weak solutions (on arbitrary time interval)
to the full equations of compressible fluid
is established \cite{FeireislBook},
and there is well-known regularity problem for these solutions.
Second, the proper treatment of nonlinear terms not only would be
hard technically but might perhaps obscure the effect of
low compressibility on the solutions, which is
the main subject of this paper. (Also note that the
linearized equations of compressible fluid are of interest on they own.
For instance, spectral properties of the operator corresponding
to linearized steady equations of compressible fluid were examined in
\cite{Pribyl}.)

Linearized equations of compressible fluid motion were
studied in \cite{IkehataKoboyashiMatsuyama,MuchaZajaczkowski}.
Estimate of strong generalized solution to initial--boundary value
problem for these equations
in a bounded three-dimensional domain
was derived in \cite{MuchaZajaczkowski},
and existence of strong generalized solution to Cauchy
problem in the whole space for them 
was established in \cite{IkehataKoboyashiMatsuyama}.
It appears that existence of weak solutions
to these equations has not been addressed.

In this paper we derive existence, uniqueness and estimates
of weak solutions to the initial--boundary value problem
for linearized equations of compressible fluid.
We examine convergence of the solutions
to the incompressible limit when compressibility
tends to zero. Briefly, we prove that
\begin{itemize}
\item in general case the velocity field converges \emph{weakly};
\item if initial condition for the velocity is solenoidal
then the velocity converges \emph{strongly} and the pressure
converges $*$-\emph{weakly};
\item if, in addition, the initial condition for the pressure
is compatible with the initial value of the pressure
in the incompressible fluid then the convergence of
the pressure is \emph{strong}.
\end{itemize}

\section{Notation and Preliminaries}
\subsection{Common Functional Spaces}
Let $E\subset \mathbb R^n$ be a bounded domain $n\in \mathbb N$.
We will use the following standard spaces:
\begin{itemize}
\item $L^p(E)$ is the Lebesgue space of real-valued functions on $E$,
      $1 \le p \le \infty$;
\item $\widehat{L}^p(E)=\setof{u}{u\in L^p(E), \ \int_E u \, dx = 0}$;
\item $H^s(E)=W^{s,p}(E)$ is the Sobolev space of real-valued
      functions whose weak derivatives up to order $s\in \mathbb N$
      belong to $L^p(E)$;
\item $\mathcal D^\prime(E)$ is the space of distributions on $E$;
      $\mathcal D(E)$ is the space of test functions on $E$;
\item $C_0^\infty(E)$ is the space of smooth real-valued
      functions on $E$ with compact support;
\item $H_0^1(E)$ denotes the closure of $C_0^\infty(E)$ in $H^1(E)$-norm;
\item $H^{-1}(E)$ denotes the dual space of $H_0^1(E)$;
\end{itemize}
(For $E=(a,b)\subset \mathbb R$ we will omit undue brackets,
i.e. $L^2(a,b)=L^2((a,b))$.)

For $\mathbb R^k$-valued functions ($k\in\mathbb N$) we will use Cartesian
products of these spaces, e.g. $L^2(E)^k$.
Since $H^{-1}(D)^k$ is linearly and continuously isomorphic
to the dual space of $H_0^1(D)^k$,
let $H^{-1}(D)^k$ \textit{denote} the latter.

Let $D\subset \mathbb R^d$ be a bounded domain with a partially-smooth
boundary $\partial D$, $d\in \mathbb N$, $d\ge 2$. Let
$Q_T = D\times (0,T)$, where $T>0$.

Vector-valued functions will be denoted by bold letters.
For such functions we will use the following spaces:
\begin{itemize}
\item $\dot{\vecsp J}(D) = 
      \setof{\vec u}{\vec u \in C_0^\infty(D)^d, \ \div \vec u=0}$;
\item $\vecsp J(D)$ is the closure 
      in $\vecsp L^2(D)$-norm of $\dot{\vecsp J}(D)$ (cf. \cite{Lady});\\
      (This space is often also denoted as $H(D)$.)
\item $\vecsp V(D)$ is the closure
      in $H_0^1(D)^d$-norm of $\dot{\vecsp J}(D)$.
\end{itemize}

\subsection{Scalar Products and Duality}\label{ss:dot-prod}

It is well-known that the space $H_0^1(D)^k$ ($k\in \mathbb N$)
is a Hilbert space with respect to the dot product
$$((\vec u, \vec v)) = 
\int_D (\nabla\otimes \vec u) : (\nabla \otimes \vec v) \, dx \equiv
\int_D (\partial_i u_k) (\partial_i v_k) \, dx$$
where $\{\vec u,\vec v\}\subset H_0^1(D)^k$
are vector-valued functions (here and further Einstein summation convention is used).

Let $(\cdot, \cdot)$, depending on the context,
denote the standard dot product in $\mathbb R^k$
or the dot product in $L^2(D)^k$, i.e. $(\vec u, \vec v)=
\int_D u_i v_i \, dx$.

Let $\|\cdot\|_X$ denote the norm of a Banach space $X$ (with dual space $X^*$)
and let $\db{\cdot,\cdot}_X$  denote
the duality brackets for the pair $(X^*,X)$.
Let $(x_n)\subset X$, $x\in X$. It is convenient to use the following notation:
\begin{itemize}
\item $x_n \wto x$ means that the sequence $(x_n)$ converges to $x$ weakly;
\item $x_n \swto x$ means that the sequence $(x_n)$ converges to $x$ $*$-weakly.
\end{itemize}

We will also use the following short-hand notation for norms (see \cite{Temam}):
$$
 \begin{aligned}
 |\cdot| &\equiv \|\cdot\|_{L^2(D)^k}\\
 \|\cdot\| &\equiv \|\cdot\|_{H_0^1(D)^k}
 \end{aligned}
$$
where the value of $k\in \mathbb N$ is defined by the context.

\subsection{Differential Operators}\label{ss:diff}
It is known that the operators $\nabla$, $\Delta$ and $\div$, 
defined on smooth functions with compact support
(i.e. functions from $C_0^\infty(D)^k$, $k\in \mathbb N$),
can be extended by continuity to bounded linear operators
$\nabla\colon L^2(D) \to H^{-1}(D)^d$,
$\Delta\colon H_0^1(D)^d \to H^{-1}(D)^d$ and
$\div\colon L^2(D)^d \to H^{-1}(D)$,
defined by
$$
\begin{aligned}
\nabla\colon & \quad  p\mapsto -(p, \div \cdot),\\
\Delta\colon & \quad  \vec u \mapsto -((\vec u, \cdot)),\\
\div\colon & \quad    \vec u \mapsto -(\vec u, \nabla \cdot).
\end{aligned}
$$

\subsection{Spaces of Banach Space Valued Functions}\label{ss:evolution}

Consider an arbitrary closed interval $[0,T]$ where $T>0$.
Let $X$ be a Banach space and $s\in \mathbb N$.
Let $L^p(0,T;X)$ and $W^{s,p}(0,T;X)$ denote accordingly
Lebesgue--Bochner and Sobolev--Bochner spaces of $X$-valued
functions of real variable $t\in [0,T]$ (see, e.g., \cite{GP}),
$1\le p \le \infty$.
Let $f_t$ denote the weak derivative of a function $f\in W^{1,p}(0,T;X)$
with respect to $t$.

Now we recall the following well-known property of Sobolev space of
Banach space valued functions (see, e.g., \cite{GP}):
\begin{prop}\label{prp:embeddingSC}
Every $u\in W^{1,2}(0,T;X)$
can be redefined on a subset $M\subset[0,T]$ of
zero measure so that $u\in C(0,T;X)$.
\emph{(Here and further $C(0,T; X)$ denotes the space of continuous functions
$f\colon [0,T] \to X$.)}
\end{prop}
\begin{rem}\label{rem:embeddingSC}
Let $\overline{u}$ denote the redefined version of $u$.
Then $\forall\varphi\in C^\infty([0,T])$, $\forall a,b\in[0,T]$
$$
\int_a^b u_t \varphi \,dt = \overline{u}\varphi|_a^b - \int_a^b u\varphi_t \,dt
$$
Moreover, the mapping
$t\mapsto \|\overline{u}(t)\|_H^2$ is absolutely
continuous with
$$
\partial_t \|\overline{u}(t)\|_H^2 = 2(u_t(t), u(t))_H.
$$
for a.e. $t\in [0,T]$.
\end{rem}

Let $X$ be a reflexive Banach space (with dual space $X^*$) and
let $H$ be a Hilbert space for which there exists a linear bounded dense
embedding $\kappa \colon X\to H$. Let $\pi \colon H \to X^*$
be the embedding given by
$\pi\colon h \mapsto (h, \kappa(\cdot))_H$,
where $(\cdot,\cdot)_H$ is the dot product in $H$.
Then embeddings $\pi$ and $\imath = \pi \circ \kappa$
are linear, bounded and dense.
Triple $(X,H,X^*)$ (with embeddings $\kappa,\pi,\imath$)
is said to be an \emph{evolution triple} \cite{GP}.
For given evolution triple let
$$
\widetilde{W}^{1,2}(0,T;X)= 
\setof{f}{f\in L^2(0,T;X), \ \imath(f) \in W^{1,2}(0,T;X^*)}.
$$
This space is referred to as \emph{Sobolev--Lions} space \cite{Sha08}.
It is a reflexive Banach space with norm given by
$$\|f\|_{\widetilde{W}^{1,2}(0,T;X)}=
\|f\|_{L^2(0,T;X)}+\|\imath(f)_t\|_{L^2(0,T;X^*)}.$$
Embedding $\imath$ is often omitted and the space
$\widetilde{W}^{1,2}(0,T;X)$ is then introduced
as the space of functions belonging to $L^2(0,T;X)$
whose weak derivative belongs to $L^2(0,T;X^*)$.
In this paper $\imath$, $\kappa$ and $\pi$
will be omitted when they are not
the subject matter.

The introduced space has the following property 
(see, e.g., \cite{GP}):
\begin{prop}\label{prp:embeddingSLC}
Every $u\in \widetilde{W}^{1,2}(0,T;X)$
can be redefined on a subset $M\subset[0,T]$ of
zero measure so that $u\in C(0,T;H)$.
\end{prop}
\begin{rem}\label{rem:embeddingSLC}
Let $\overline{u}$ denote the redefined version of $u$.
Then $\forall a,b\in[0,T]$,
$\forall v\in X$ and $\forall\varphi\in C^\infty([0,T])$
$$
\int_a^b \db{u_t, v} \varphi \, dt =
(\overline{u}(t)\varphi(t), v)_H \Big|_{a}^{b}
 - \int_a^b (u, v)_H \varphi_t \, dt.
$$
Consequently, \emph{if 
$\forall v\in X$ and $\forall\varphi\in C_0^\infty([0,T))$
$$
\int_0^T \db{u_t, v} \varphi \, dt =
- (u_0\varphi(0), v)_H
 - \int_0^T (u, v)_H \varphi_t \, dt,
$$
where $u_0\in H$, then $\overline{u}(0)=u_0$.}
(Similar procedure can be used to identify the initial value of
a function from $W^{1,2}(0,T;H)$.)
\end{rem}
\begin{rem}\label{rem:embeddingSLC:1}
The mapping $t\mapsto \|\overline{u}(t)\|_H^2$ is absolutely
continuous with
$$
\partial_t \|\overline{u}(t)\|_H^2 = 2\db{u_t(t), u(t)}
$$
for a.e. $t\in [0,T]$.
\end{rem}

In this paper we consider evolution triples
$(\vecsp H_0^1(D)^k, \vecsp L^2(D)^k, \vecsp H^{-1}(D)^k)$
($k\in \mathbb{N}$) and 
$(\vecsp V(D), \vecsp J(D), \vecsp V(D)^*)$.
For both of them the embeddings $\kappa$, $\pi$ and $\imath$
are given by $\kappa \colon \vec u\mapsto \vec u$ (natural embedding),
$\pi \colon \vec u \mapsto \int_D(\vec u, \cdot) \, dx$
and $\imath=\pi\circ\kappa$.

The following theorem describes the space which is dual to a Lebesgue--Bochner space \cite{GP}:
\begin{prop}\label{prp:dual}
If $X$ is a reflexive Banach space, then
$$L^p(0,T;X)^* = L^{p'}(0,T;X^*),$$
where $1\le p < \infty$, $1/p+1/p'=1$ and
duality is given by $\db{f,g} = \int_0^T fg \,dt$, $f\in L^{p'}(0,T;X^*)$,
$g\in L^p(0,T;X)$.
\end{prop}

Different constants which are not dependent on the
principal parameters (such as initial conditions)
will be denoted by the same letter $C$.
The dependence of such constant on some parameter
will be indicated in the subscript.

\subsection{Auxiliary Inequalities}
Let us recall two well-known statements:
\begin{prop}\label{prp:ineq-convex}
Let $a\ge 0$, $b\ge 0$ and $J$ be real numbers. If
$J^2 \le a+ b J$ then $J \le b + \sqrt{a}$.
\end{prop}
\begin{prop}[Gronwall's inequality]\label{prp:Gronwall}
Let $I$ be an absolutely continuous nonnegative function of a variable $t\in[0,T]$
and let $\varphi, \psi \in L^1(0,T)$ be nonnegative functions.
If the derivative $I'(t)$ satisfies
$$
I'(t) \le \varphi(t) I(t) + \psi(t) \quad \text{a.e. on } [0,T],
$$
then for a.e. $t\in [0,T]$
$$
I(t) \le e^{\int_0^t \varphi(\tau) \, d\tau}\left(
I(0) + \int_0^t \psi(\tau) \, d\tau \right).
$$
\end{prop}
The proof of Proposition \ref{prp:ineq-convex} is elementary and
the proof of Proposition \ref{prp:Gronwall} can be found e.g. in \cite{Evans}.

We will use the following mix of 
Lemmas \ref{prp:ineq-convex} and \ref{prp:Gronwall}:

\begin{lem}\label{lem:mixed}
Let $I$ and $J$ be absolutely continuous nonnegative functions of a variable $t\in[0,T]$,
$J\in L^2(0,T)$. Let $a, c \in L^1(0,T)$ and $b\in L^2(0,T)$
be nonnegative functions. If for a.e. $t\in[0,T]$
\begin{equation}\label{eq:mixed}
I'(t) + J^2(t) \le a(t) I(t) + b(t) J(t) + c(t)
\end{equation}
then
\begin{gather*}
\|J\|_{L^2(0,T)} \le C_a \left(
 \sqrt{I(0)} + \sqrt{\|c\|_{L^1(0,T)}} + \|b\|_{L^2(0,T)}
 \right), \\
 \|I\|_{L^\infty(0,T)} \le C_a \left(
 I(0) + \|c\|_{L^1(0,T)} + \|b\|^2_{L^2(0,T)}
 \right),
\end{gather*}
where constant $C_a$ depends only on $A=\|a\|_{L^1(0,T)}$
\end{lem}

\begin{proof}
From \eqref{eq:mixed} for a.e. $t\in [0,T]$
\begin{equation*}
I'(t) \le a(t) I(t) + b(t) J(t) + c(t)
\end{equation*}
and then, by Proposition \ref{prp:Gronwall} and Cauchy--Bunyakovsky inequality,
\begin{multline*}
I(t) \le \exp\left(\int_0^t a \, d\tau \right)
\left( I(0) + \int_0^t (bJ+c) d\tau \right)
\le\\ \le
e^{\|a\|_{L^1(0,t)}}
\left(
I(0) + \|b\|_{L^2(0,t)} \|J\|_{L^2(0,t)} + \|c\|_{L^1(0,t)}
\right)
\end{multline*}
and hence
\begin{equation*}
\|I\|_{L^\infty(0,T)} \le 
e^{\|a\|_{L^1(0,T)}}
\left(
\|b\|_{L^2(0,T)} \|J\|_{L^2(0,T)} + \|c\|_{L^1(0,T)} + I(0)
\right).
\end{equation*}
Integrating the inequality $J^2\le aI+bJ+c - I'$ and noting that $I(T)\ge 0$ we obtain
\begin{multline}\label{eq:J2ineq}
\int_0^T J^2 dt \le \int_0^T aI \, d\tau + \int_0^T bJ \, d\tau + \int_0^T c\,
d\tau + I(0) \le \\
\le \|I\|_{L^\infty(0,T)} \|a\|_{L^1(0,T)} + 
\|b\|_{L^2(0,T)}\|J\|_{L^2(0,T)} + \|c\|_{L^1(0,T)} + I(0) \le \\
\le C_a \left( \|b\|_{L^2(0,T)}\|J\|_{L^2(0,T)} + \|c\|_{L^1(0,T)} + I(0)\right)
\end{multline}
where
$$
C_a = 1 + \|a\|_{L^1(0,T)} e^{\|a\|_{L^1(0,T)}}.
$$
Applying Proposition \ref{prp:ineq-convex} to \eqref{eq:J2ineq} we get
$$
\|J\|_{L^2(0,T)} \le C_a \left(\sqrt{\|c\|_{L^1(0,T)}} + \sqrt{I(0)} +
 \|b\|_{L^2(0,T)}\right)
$$
Finally, by Young's inequality
\begin{multline*}
\|I\|_{L^\infty(0,T)} \le e^{\|a\|_{L^1(0,T)}}
\left(
\frac{\|b\|_{L^2(0,T)}^2 + \|J\|_{L^2(0,T)}^2}{2} + \|c\|_{L^1(0,T)} + I(0)
\right)\le \\
\le e^{\|a\|_{L^1(0,T)}}
\left(
\|b\|_{L^2(0,T)}^2 + \|c\|_{L^1(0,T)} + I(0)+{}\right.\\
\left.{}+ \frac32 C_a^2 \left(\|b\|_{L^2(0,T)}^2 + \|c\|_{L^1(0,T)} + I(0)\right)
\right) ={} \\
{}= \tilde C_a \left(\|b\|_{L^2(0,T)}^2 + \|c\|_{L^1(0,T)} + I(0)\right),
\end{multline*}
where
\begin{equation*}
\tilde C_a = e^{\|a\|_{L^1(0,T)}}\left(1 + \frac32 C_a^2\right). \qedhere
\end{equation*}
\end{proof}

\subsection{Properties of Special Subspaces}
Let $G(D)$ denote the orthogonal complement of $J(D)$ in $L^2(D)^d$.
Let $P_J$ and $P_G$ denote the orthogonal projectors of $L^2(D)^d$
onto $J(D)$ and $G(D)$ respectively. Projector $P_J$
is referred to as \emph{Leray projector}.

Let 
$$V(D)^\perp = \setof{f}{f\in H^{-1}(D)^d,
\ f(\vec u)=0, \ \forall \vec u \in V(D)}.$$
It is clear that
functionals represented by $\nabla p$, $p\in L^2(D)$,
belong to $V(D)^\perp$. To show that any functional from
the latter space can be represented is such form,
the following statements are needed
(see \cite{AS} and \cite[I, \S 1]{Temam}): 

\begin{prop}\label{prp:grad-repr}
A functional $f\in H^{-1}(D)^d$ is representable in the form
$$f = \nabla p$$ for some $p\in L^2(D)$ if and only if
$$\db{f,\vec u} = 0 \quad \forall \vec u \in \dot{J}(D).$$
\end{prop}

\begin{prop}\label{prp:grad-est}
If $p\in \widehat{L}^2(D)$ then
$$\|p\|_{\widehat{L}^2(D)} \le C \|\nabla p\|_{H^{-1}(D)^d},$$
where the constant $C$ depends only on the domain $D$.
\end{prop}

Propositions \ref{prp:grad-repr} and \ref{prp:grad-est} 
allow us to introduce an operator, which is inverse to $\nabla$:
\begin{lem}\label{lem:nabla-inv}
The operator $\nabla\colon \widehat{L}^2(D) \to V(D)^\perp$
has a bounded inverse $\nabla^{-1}$.
\end{lem}
\begin{proof}
For $f\in V(D)^\perp$ let $p_1=p-\int_D p \, dx$, where
$p$ is given by Proposition \ref{prp:grad-repr}. If there exists
$p_2 \in \widehat{L}^2(D)$ such that $f=\nabla p_2$ then
$\nabla(p_1-p_2)=0$ and from Proposition \ref{prp:grad-est}
$p_1-p_2=0$, therefore we can introduce the operator
$\nabla^{-1}\colon f \mapsto p_1$ (which is clearly linear).

It follows from Proposition \ref{prp:grad-est} that
$\|p_1\|_{\widehat{L}^2(D)} \le C \|f\|_{H^{-1}(D)}$
and therefore $\nabla^{-1}$ is bounded.
\end{proof}

Later we will use the operator which was first
constructed by M.E. Bogovskii \cite{Bogovskii}:

\begin{prop}\label{prp:bogovskii}
Suppose that the domain $D$ is star-shaped with respect to some ball.
Then there exists a bounded linear operator 
$\mathcal B\colon \widehat{L}^2(D)\to H_0^1(D)^d$
such that $\forall f\in \widehat{L}^2(D)$
the field $\vec u = \mathcal B(f)$ satisfies
$$
\div \vec u = f \quad \text{(a.e. in $D$).}
$$
\end{prop}

\section{Linearized Equations of Low Compressible Fluid Motion}
Consider viscous barotropic fluid with state equation $\rho = F(p)$, where
$\rho$ and $p$ are density and pressure respectively. Let us linearize this
equation near some reference pressure $p_{ref}$:
$$
\rho=F(p) \approx F(p_{ref}) + F'(p_{ref})(p-p_{ref}).
$$
For brevity assume that $p_{ref}=0$ and denote $F(0)=\rinc>0$, $F'(0)=\comp$.
From the physical point of view $1/\comp = \frac{\partial p }{\partial \rho} = c^2 >0$,
where $c$ is the speed of sound. Ultimately we have $\rho = \rinc + \comp p.$

We are going to study the following equations:
\begin{equation}\label{eq:comp}
\begin{gathered}
\rho_t + \rho_0 \div \vec u = 0,\\
\rinc\vec u_t + \nabla p = \mu \Delta \vec u + \eta \nabla \div \vec u + \rho \vec f,\\
\rho = \rinc+\comp p.
\end{gathered}
\end{equation}
Here
\begin{list}{}{}
\item $\rho = \rho (x,t)$ \quad ($\rho(x,t)\in \mathbb R$), \qquad
      $p=p(x,t)$ \quad ($p(x,t) \in \mathbb R$),
\item $x\in D \subset (0,T)$,
 $t\in (0,T)\subset \mathbb R$,
\item $\vec u = \vec u(x,t)$ is the velocity ($\vec u(x,t)\in \mathbb R^d$),
\item $\vec f = \vec f(x,t)$ is the external force (per unit volume).
\end{list}
The coefficients of viscosity $\mu > 0$ and $\eta \ge 0$ are constant
throughout the paper.
The constant $\comp>0$ is referred to as the \emph{compressibility factor}
\cite{Shifrin}.

Equations \eqref{eq:comp} arise as the linearisation of the equations of compressible
fluid motion near some solution $\{\uinc, \pinc\}$ to the corresponding equations
of incompressible fluid.
The convective terms $(\uinc, \nabla)\comp p$
and $(\uinc, \nabla)\vec u$ are omitted for the sake of technical simplicity.

Let us consider the following initial--boundary conditions for the equations \eqref{eq:comp}:
\begin{equation}\label{eq:ocomp}
\vec u|_{\partial D} = 0, \quad \vec u|_{t=0} = \vec u_0, \quad p|_{t=0}=p_0.
\end{equation}
It is clear that the solution to the problem \eqref{eq:comp}, \eqref{eq:ocomp}
depends on the compressibility factor $\comp$:
$$\{\vec u, p\} = \{\vec u_\comp , p_\comp\}.$$
Our goal is to study the convergence of these solutions when
compressibility tends to zero, i.e. $\comp \to 0$.
But before doing this, we need to state (and prove) the
existence and uniqueness theorems for the problem \eqref{eq:comp}, \eqref{eq:ocomp}
and for the corresponding incompressible system.

\subsection{Existence and Uniqueness of Weak Solutions}

Consider a non-homogeneous form of the problem \eqref{eq:comp}, \eqref{eq:ocomp}:
\begin{gather}
\begin{gathered}\label{eq:gcomp}
\rho_t + \rho_0 \div \vec u = \sigma, \quad \rho = \comp p,\\
\rinc\vec u_t + \nabla p = \mu \Delta \vec u + \eta \nabla \div \vec u + \rho \vec f + \vec s,\\
\end{gathered}\\
\begin{gathered}\label{eq:ogcomp}
\vec u|_{\partial D} = 0, \quad \vec u|_{t=0} = \vec u_0, \quad p|_{t=0}=p_0.
\end{gathered}
\end{gather}
Setting $\sigma=0$ and $\vec s= \rinc \vec f$ we clearly obtain 
the problem \eqref{eq:comp}, \eqref{eq:ocomp}.

Assume that 
$$
\begin{gathered}
\sigma = L^2(0,T;L^2(D)), \quad \vec s \in L^2(0,T;H^{-1}(D)^d), \\
\vec f \in L^\infty (Q_T)^d, \quad \vec u_0 \in L^2(D)^d, \quad p_0\in L^2(D).
\end{gathered}
$$

\begin{defn}\label{def:weak1}
A pair 
$$\{\vec u, p\} \in L^2(0,T;H_0^1(D)^d)\times L^2(0,T;L^2(D))$$
is called a \emph{weak solution} to the problem \eqref{eq:gcomp}, \eqref{eq:ogcomp},
if for all $\varphi \in C_0^\infty([0,T))$, $\vec v \in H_0^1(D)^d$ and
$r\in L^2(D)$
\begin{gather}\label{eq:def:weak1:1}
-\int_0^T ( \rho, r ) \, \varphi_t \, dt - (\comp p_0, r) \varphi(0)
+ \int_0^T (\rho_0 \div \vec u, r ) \, \varphi \, dt =
 \int_0^T ( \sigma, r ) \, \varphi \, dt,
\end{gather}
\begin{multline}\label{eq:def:weak1:2}
-\int_0^T \rinc (\vec u, \vec v) \, \varphi_t \, dt - \rinc (\vec u_0, \vec v) \varphi(0)
-\int_0^T ( p, \div \vec v ) \, \varphi \, dt = {}\\
{} = -\mu \int_0^T ((\vec u, \vec v)) \, \varphi \, dt
- \eta \int_0^T (\div \vec u, \div \vec v) \, \varphi \, dt + {}\\
{} + \int_0^T (\rho \vec f, \vec v) \, \varphi \, dt +
\int_0^T \db{\vec s, \vec v} \varphi \, dt,
\end{multline}
where $\rho = \comp p$.
(See subsection \ref{ss:dot-prod} for the notation.)
\end{defn}
\begin{rem}\label{rem:eqae}
From \eqref{eq:def:weak1:1} and \eqref{eq:def:weak1:2} it
follows that
$$
\begin{gathered}
\vec u \in \widetilde{W}^{1,2}(0,T;H_0^1(D)^d),\\
p\in W^{1,2}(0,T;L^2(D))
\end{gathered}
$$
and the equations \eqref{eq:gcomp}
hold for a.e. $t\in[0,T]$
in sense of notation, introduced in subsections \ref{ss:diff} and \ref{ss:evolution}.
Consequently, for a.e. $t\in[0,T]$
$$
\begin{gathered}
(\rho_{t} + \rho_0\div \vec u, p) = (\sigma, p), \\
\rinc\db{\vec u_{t},\vec u} - (p, \div \vec u) = \db{\mu \Delta \vec u
+ \eta \nabla \div \vec u + \rho \vec f + \vec s, \vec u} \\
\end{gathered}
$$

Due to Propositions \ref{prp:embeddingSC} and \ref{prp:embeddingSLC} 
the functions $\vec u$ and $p$ can be redefined so that they
have well-defined values $\overline{\vec u}(t)$ and $\overline{p}(t)$
at each $t\in [0,T]$.
Integrating the equations above
with respect to $t$ using Remarks \ref{rem:embeddingSC} and \ref{rem:embeddingSLC:1}
we obtain the \emph{energy equality}:
\begin{multline}\label{eq:energy}
\frac12 \left.\left(\rinc|\overline{\vec u}|^2 + \frac{\alpha}{\rinc}
 |\overline{p}|^2 \right)\right|_\xi^\tau +
 \int_\xi^\tau (\mu\|\vec u\|^2  + \eta |\div \vec u|^2) \, dt ={} \\
{}= \int_\xi^\tau (\comp p \vec f, \vec u) \, dt + \int_\xi^\tau
 \db{\vec s, \vec u} dt
 + \int_\xi^\tau \frac{1}{\rinc}(\sigma, p) \, dt,
\end{multline}
where $\xi,\tau\in[0,T]$ are arbitrary.

By Remark \ref{rem:embeddingSLC}, from
equations \eqref{eq:def:weak1:1}, \eqref{eq:def:weak1:2}
we conclude that $\overline{\vec u}(0) = \vec u_0$ and
$\overline{p}(0)=p_0$.
\end{rem}

Definition \ref{def:weak1} is equivalent
to the following one:

\begin{defn}\label{def:weak2}
A pair 
$$\{\vec u, p\} \in L^2(0,T;H_0^1(D)^d)\times L^2(0,T;L^2(D))$$
is called a \emph{weak solution} to the problem \eqref{eq:gcomp}, \eqref{eq:ogcomp},
if the equations \eqref{eq:gcomp} hold in sense of distributions on $Q_T$,
and the initial values of $\vec u$ and $p$ are $\vec u_0$ and $p_0$
respectively.
\end{defn}
\begin{proof}[Proof of equivalence of \ref{def:weak1} and \ref{def:weak2}]
To prove the implication \ref{def:weak2} $\Rightarrow$ \ref{def:weak1}
consider $\vec v\in H_0^1(D)^d$ and $r\in L^2(D)$. There exist
$(\vec v_n) \subset C_0^\infty(D)^d$ and $(r_n) \subset C_0^\infty(D)$
such that $\vec v_n \to \vec v$ and $r_n \to r$ in $H_0^1(D)^d$
and $L^2(D)$ respectively. Then, for each $\varphi\in C_0^\infty (0,T)$
$$\begin{gathered}
\vec \Phi_n(x,t)=\vec v_n(x) \varphi(t) \in \mathcal D(Q_T)^d,\\
\Phi_n(x,t)=r_n(x) \varphi(t) \in \mathcal D(Q_T).
\end{gathered}$$
Integrating by parts the distributional formulation of the equations \eqref{eq:gcomp}
and passing to the limit when $n\to \infty$
we obtain \eqref{eq:def:weak1:1} and \eqref{eq:def:weak1:2}
with $\varphi\in C_0^\infty (0,T)$. Then, from
Propositions \ref{prp:embeddingSC} and \ref{prp:embeddingSLC}
it follows that \eqref{eq:def:weak1:1} and \eqref{eq:def:weak1:2}
hold with arbitrary $\varphi\in C_0^\infty ([0,T))$.

The implication \ref{def:weak1} $\Rightarrow$ \ref{def:weak2} follows
from the fact that the set
$$
\setof{\sum_{k=1}^N \varphi_k(t) r_k(x)}{N\in \mathbb N, \
 \varphi_k\in \mathcal D(0,T), \ r_k \in \mathcal D(D), \ k=1..N}
$$
is dense in $\mathcal D(Q_T)$ (see e.g. \cite{VS}).
\end{proof}

Now we are ready to state the main result of this section:
\begin{thm}\label{th:exist:comp}
The problem \eqref{eq:gcomp}, \eqref{eq:ogcomp} has a unique weak solution $\{\vec u, p\}$.
For this solution the following estimates are valid:
\begin{gather}
\|\vec u\|_{L^2(0,T;H_0^1(D)^d)} +
\|\vec u\|_{L^\infty(0,T;L^2(D)^d)} + {} \notag\\
{}+ \sqrt{\comp} \|p\|_{L^\infty(0,T; L^2(D))}
\le C_{\vec f, \comp} E,\label{eq:est1}\\
\|\vec u\|_{\widetilde{W}^{1,2}(0,T;H_0^1(D)^d)} \le
\frac{1}{\sqrt{\comp}} \tilde{C}_{\vec f, \comp} E,\label{eq:est2}
\end{gather}
where
\begin{multline*}
E \equiv \|\vec u_0\|_{L^2(D)^d} + \sqrt{\comp}\|p_0\|_{L^2(D)} + {}\\
{}+\frac{1}{\sqrt{\comp}}\|\sigma\|_{L^2(0,T;L^2(D))} + \|\vec s\|_{L^2(0,T;H^{-1}(D)^d)}
\end{multline*}
and the constants $C_{\vec f, \comp}, \tilde{C}_{\vec f, \comp}$
depend only on $\|\vec f\|_{L^\infty(Q_T)^d}$ and $\comp$. 
For fixed $\vec f$ these constants are bounded when $\comp \to 0$.
Moreover, if $p_0\in \widehat{L}^2(D)$ and 
$\sigma \in L^2(0,T;\widehat{L}^2(D))$, then 
$$
p\in W^{1,2}(0,T;\widehat{L}^2(D)).
$$
\end{thm}
\begin{proof}
We will use Faedo--Galerkin method. Let $\{\vec e_j\}_{j=1}^\infty$ be an orthonormal
with respect to $((\cdot,\cdot))$ basis in $H_0^1(D)^d$ and let $\{e_j\}_{j=1}^\infty$
be an orthonormal basis in $L^2(D)$. Then,
since $H_0^1(D)^d$ is dense in $L^2(D)$, for $m\in \mathbb N$ there exist
$$
\begin{gathered}
p^0_m = \sum_{j=1}^m \tilde c_{0,j} e_j, \quad \tilde c_{0,j}=(p_0, e_j),\\
\vec u^0_m = \sum_{j=1}^m c^{(m)}_{0,j} \vec e_j, \quad c^{(m)}_{0,j}\in \mathbb R,
\end{gathered}
$$
such that $p^0_m \to p_0$ in $L^2(D)$ and $\vec u^0_m \to \vec u_0$
in $L^2(D)^d$
when $m\to \infty$.

Let us look for the solution in the form
$$
\vec u_m = \sum_{j=1}^{m} c^{(m)}_j(t) \vec e_j, \quad
p_m = \sum_{j=1}^{m} \tilde c^{(m)}_j(t) e_j, 
$$
where the functions $c^{(m)}_j(t)$ and $\tilde c^{(m)}_j(t)$ are absolutely continuous.
Consider the following ``approximate'' system:
$$
\begin{gathered}
(\rho_{m,t} + \rinc \div \vec u_m, e_i) = (\sigma, e_i) \ \text{ a.e. on } [0,T], \\
\db{\rinc\vec u_{m,t} + \nabla p_m, \vec e_i} = \db{\mu \Delta \vec u_m
+ \eta \nabla \div \vec u_m + \rho_m \vec f + \vec s, \vec e_i} \ \text{ a.e. on } [0,T], \\
i= 1..m, \quad \rho_m = \comp p_m, \\
\vec u_m|_{t=0} = \vec u^0_m, \quad p_m|_{t=0} = p^0_m.
\end{gathered}
$$
This system represents a Cauchy problem for a linear system of
first order differential equations. Under
our assumptions for each $m$ it has a unique solution 
$\{c^{(m)}_j(t), \tilde c^{(m)}_j(t)\}_{j=1}^m$
such that 
$$\{u_m, p_m\} \in
W^{1,2}(0,T;H_0^1(D)^d) \times W^{1,2}(0,T;L^2(D)).$$

Since $\vec u_m(t) \in \mspan \{\vec e_j\}_{j=1}^m$
and $p_m(t) \in \mspan \{e_j\}_{j=1}^m$, the following equations 
hold a.e. on $[0,T]$:
$$
\begin{gathered}
(\rho_{m,t} + \rinc \div \vec u_m, p_m) = (\sigma, p_m), \\
(\rinc\vec u_{m,t},\vec u_m) - (p_m, \div \vec u_m) = \db{\mu \Delta \vec u_m
+ \eta \nabla \div \vec u_m + \rho_m \vec f + \vec s, \vec u_m} \\
\end{gathered}
$$
Hence, recalling the notation from subsection \ref{ss:dot-prod},
by Remarks \ref{rem:embeddingSC} and \ref{rem:embeddingSLC:1}
\begin{multline}\label{eq:energy-m}
\frac12 \left(\rinc|\vec u_m|^2 + \frac{\alpha}{\rinc} |p_m|^2 \right)_t +
 \mu \|\vec u_m\|^2 + \eta |\div \vec u_m|^2 ={} \\
{}= (\comp p_m \vec f, \vec u_m) + \db{\vec s, \vec u_m} + 
\frac{1}{\rinc}(\sigma, p_m).
\end{multline}

Let us make some auxiliary estimates. First, from Young's inequality,
$$
\frac{\sqrt{\comp}}{\rinc} (\sqrt{\comp} p_m \vec f, \rinc \vec u_m)
\le \frac12 \frac{\sqrt{\comp}}{\rinc} \|\vec f\|_{L^\infty (Q_T)^d}
\left(\comp |p_m|^2 + \rinc^2|\vec u_m|^2\right)
$$
and $(\sigma, p_m) \le \frac12 \left(\comp|p_m|^2 + \frac{1}{\comp} |\sigma|^2\right)$.
Second, $\db{\vec s, \vec u_m} \le \|\vec s\|_{H^{-1}(D)^d}\|\vec u_m\|$.
Then, denoting
$$
\begin{gathered}
I(t) = \frac12 \left(\rinc|\vec u_m(t)|^2 + \frac{\alpha}{\rinc}
 |p_m(t)|^2\right),
 \quad J(t) = \sqrt{\mu} \|\vec u_m(t)\|,\\
 a= 1 + \sqrt{\comp} \|\vec f\|_{L^\infty(Q_T)^d}, \quad
 b(t)= \frac{1}{\sqrt{\mu}}\|\vec s(t)\|_{H^{-1}(D)^d}, \quad
 c(t)= \frac{1}{2\rinc\alpha} |\sigma(t)|^2 
\end{gathered}
$$
from \eqref{eq:energy-m} we obtain
$$
I' + J^2 \le a I + b J + c
$$
From this inequality, by Lemma \ref{lem:mixed}, we have
\begin{multline*}
\|J\|_{L^2(0,T)} \le C_{\vec f, \comp} \left(
 |\vec u^0_m| + \sqrt{\comp} |p^0_m| + \right.{} \\
\left.{}  + \frac{1}{\sqrt{\comp}} \|\sigma\|_{L^2(0,T;L^2(D))}  +
  \|\vec s\|_{L^2(0,T;H^{-1}(D)^d)}
 \right),
\end{multline*}
\begin{multline*}
 \|I\|_{L^\infty(0,T)} \le C_{\vec f, \comp} \left(
  |\vec u^0_m|^2 + \comp |p^0_m|^2 + {}\right. \\
 \left.{}  + \frac{1}{\comp} \|\sigma\|_{L^2(0,T;L^2(D))}^2  +
   \|\vec s\|_{L^2(0,T;H^{-1}(D)^d)}^2
  \right).
\end{multline*}
From the proof of Lemma \ref{lem:mixed} it follows that
the constant $C_{\vec f, \comp}$ depends only on $\|\vec f\|_{L^\infty(Q_T)^d}$
and $\alpha$, and for each $\vec f$ this constant is bounded when $\comp\to 0$.

Thus we have proved that estimate \eqref{eq:est1} holds for the
``approximate solutions'' $\vec u_m$, $p_m$. Since any weak solution
to the problem \eqref{eq:gcomp}, \eqref{eq:ogcomp}
satisfies the energy equality \eqref{eq:energy}, the arguments
above imply that \eqref{eq:est1} holds for each weak solution
to the problem \eqref{eq:gcomp}, \eqref{eq:ogcomp} as well.

By Alaoglu--Bourbaki theorem (and Proposition \ref{prp:dual}),
there exists a subsequence (for brevity not renumbered) such that
$$
\begin{gathered}
\vec u_m \wto \vec u \text{\; in } L^2(0,T;H_0^1(D)^d), \\
\vec u_m \swto \vec u \text{\; in } L^\infty(0,T;L^2(D)^d), \\
p_m \swto p \text{\; in }  L^\infty(0,T;L^2(D))
\end{gathered}
$$
and the limits $\vec u$ and $p$ satisfy \eqref{eq:est1}.

Now let $\varphi\in C_0^\infty([0,T))$ and for $r\in L^2(D)$ denote
$$
F_m(r) \equiv \int_0^T \left(-\rho_m \varphi_t + (\rinc\div \vec u_m -
\sigma)\varphi, r \right ) dt - (\comp p_m(0), r) \varphi(0),
$$
It is clear that $(F_m)$ is bounded in $L^2(D)^*$ and for
each $r\in \mspan \{e_j\}_{j=1}^m$ $F_m(r)=0$.
Then, since $\{e_j\}_{j=1}^\infty$ is an orthonormal basis of $L^2(D)$,
$$\lim_{m\to\infty} F_m(r) = 0.$$
On the other hand,
$$\lim_{m\to\infty} F_m(r)= F(r) \equiv \int_0^T \left(-\rho \varphi_t +
(\rinc \div \vec u -
\sigma)\varphi, r \right ) \, dt - (\comp p_0, r) \varphi(0),$$
where $\rho = \comp p$. Thus we have shown that \eqref{eq:def:weak1:1}
holds. Similar arguments for the sequence of the functionals
\begin{multline*}
G_m(\vec v) \equiv \int_0^T \left(-\rinc\vec u_m \varphi_t, \vec v\right) dt
 - \int_0^T ( p_m, \div \vec v) \varphi \, dt
 - \rinc(\vec u_m(0), \vec v) \varphi(0) -{}\\{}-
 \int_0^T \db{\mu \Delta \vec u_m + \eta \nabla \div \vec u_m +
  \rho_m \vec f + \vec s, \vec v} \varphi \, dt, \qquad \vec v\in H_0^1(D)^d
\end{multline*}
imply that \eqref{eq:def:weak1:2} holds.

From \eqref{eq:def:weak1:1} and \eqref{eq:def:weak1:2} we conclude that
a.e. on $[0,T]$
$$
\begin{gathered}
|p_t| \le \frac{1}{\comp} \left(|\rinc\div \vec u| + |\sigma|\right),\\
\|\vec u_t\|_{H^{-1}(D)^d} \le \|-\nabla p+ \mu \Delta \vec u +
 \eta \nabla \vec u+\comp p \vec f + \vec s\|_{H^{-1}(D)^d}/\rinc,
\end{gathered}
$$
hence 
$\vec u$ satisfies \eqref{eq:est2}.

If the problem \eqref{eq:gcomp}, \eqref{eq:ogcomp} had two weak solutions
$\{\vec u^1, p^1\}$ and $\{\vec u^2, p^2\}$, then the difference
$\{\vec u^1 - \vec u^2, p^1-p^2\}$ would be a weak solution to the
problem \eqref{eq:gcomp}, \eqref{eq:ogcomp} with zero
$\vec u_0$, $p_0$, $\sigma$ and $\vec s$. Hence, from
estimate \eqref{eq:est1} we would have
$\vec u^1 - \vec u^2=0$ and $p^1-p^2=0$. Consequently, the 
problem \eqref{eq:gcomp}, \eqref{eq:ogcomp} has
only one weak solution.

Finally, if $p_0\in \widehat{L}^2(D)$ and $\sigma \in L^2(0,T;\widehat{L}^2(D))$,
then taking $\{e_j\}_{j=1}^\infty$ as an orthogonal basis of $\widehat{L}^2(D)$
we obtain $p\in W^{1,2}(0,T;\widehat{L}^2(D))$.
\end{proof}

\begin{cor}\label{cor:exist:comp}
Suppose that 
$$\vec f \in L^\infty (Q_T)^d, \quad \vec u_0 \in L^2(D)^d, \quad p_0\in L^2(D).$$
Then the problem \eqref{eq:comp}, \eqref{eq:ocomp} has a unique weak solution $\{\vec u, p\}$.
\end{cor}

\subsection{Incompressible Limit}
When $\comp=0$, equations \eqref{eq:ocomp} formally turn into
nonsteady Stokes equations for incompressible fluid:

\begin{equation}\label{eq:inc}
\begin{gathered}
\div \uinc = 0,\\
\rinc\uinc_t + \nabla \pinc = \mu \Delta \uinc + \rinc \vec f.
\end{gathered}
\end{equation}
Consider the following initial--boundary conditions for the equations \eqref{eq:inc}:
\begin{equation}\label{eq:oinc}
\begin{gathered}
\uinc|_{\partial D} = 0,\quad
\uinc|_{t=0} = \uinco.
\end{gathered}
\end{equation}

The problem \eqref{eq:inc}, \eqref{eq:oinc} is well-studied, so we will only
state some standard results for it.
\begin{defn}
A pair 
$$\{\uinc, \pinc\} \in L^2(0,T;V(D))\times \mathcal D'(Q_T)$$
is called a \emph{weak solution} to the problem \eqref{eq:inc}, \eqref{eq:oinc},
if \eqref{eq:inc} holds in $\mathcal D'(Q_T)$ and
for all $\varphi \in C_0^\infty([0,T))$ and $\vec v \in V(D)$
\begin{multline}\label{eq:def:weak:3}
-\int_0^T \rinc(\uinc, \vec v) \, \varphi_t \, dt - \rinc(\uinco, \vec v) \varphi(0)
= {}\\
{} = -\mu \int_0^T ((\uinc, \vec v)) \, \varphi \, dt
 + \int_0^T (\rinc \vec f, \vec v) \, \varphi \, dt
\end{multline}
\end{defn}

\begin{prop}\label{th:exist:inc}
Let $\vec f \in L^2 (0,T;H^{-1}(D)^d)$, $\uinco \in J(D)$.
Then the problem \eqref{eq:inc}, \eqref{eq:oinc} has a unique weak solution
such that
$$
\begin{gathered}
\uinc \in \widetilde{W}^{1,2}(0,T; \vecsp V(D)),\\
\pinc = \partial_t Q, \quad Q \in C(0,T; \widehat{L}^2(D)).\\
\end{gathered}
$$
Moreover, if $\uinco \in V(D)$ and $\vec f\in L^2(0,T; L^2(D))$,
then
$$
\begin{gathered}
\uinc_t \in L^2(0,T; J(D)),\\
\pinc \in L^2(0,T; \widehat{L}^2(D)).\\
\end{gathered}
$$
\end{prop}
\begin{rem}
The first part of the proposition is proved in \cite{Temam}.
The case $\uinco \in V(D)$ is treated by a minor modification
of the proof from \cite{Temam}.
\end{rem}
\begin{rem}
Every weak solution to the problem \eqref{eq:inc}, \eqref{eq:oinc}
satisfies the following \emph{energy equality}:
\begin{equation}\label{eq:energy-inc}
\rinc|\overline{\uinc}|^2 |_\xi^\tau + 2\mu \int_\xi^\tau \|\uinc\|^2 \, dt
= 2\int_\xi^\tau (\rinc \vec f, \uinc)\, dt
\end{equation}
for all $\xi,\tau \in [0,T]$.
\end{rem}

Regularity of the solutions to the problem \eqref{eq:inc},
\eqref{eq:oinc} has been rigorously studied by V.A. Solonnikov.
Let us state a corollary of one of his results (see \cite[p. 126]{Lady}):

\begin{prop}[V.A. Solonnikov]\label{th:exist:inc:reg}
Let $\ell$ be a nonnegative integer number.
Let $\partial D \in C^{2\ell + 2}$,
$\vec f \in W_2^{2\ell, \ell} (Q_T)^d$
and $\uinco \in J(D) \cap W_2^{2\ell + 2}(D)^d$.
Assume that all the necessary 
compatibility conditions up to order $\ell$ are satisfied.
Then
$$
\uinc \in W^{2\ell+2, \ell + 1}(Q_T)^d, \quad
\nabla \pinc \in W_2^{2\ell, \ell}(Q_T)^d.
$$
\end{prop}
\begin{rem}
It follows from Proposition \ref{th:exist:inc:reg} that
$\nabla \pinc \in W^{\ell,2}(0,T;L^2(\Omega))$.
Applying $\nabla^{-1}$ to $\nabla \pinc$ we get
$\pinc \in W^{\ell,2}(0,T; \widehat{L}^2(\Omega))$.
Then for $\ell \ge 1$
$$\pinc|_{t=0}=\nabla^{-1} P_{\vecsp G}(\Delta \uinco + \vec f(0)),$$
i.e. the pressure $\pinc$ has a well-defined initial value.
\end{rem}

\section{Convergence Towards the Incompressible Limit}
In this section we suppose that
 the assumptions of Corollary~\ref{cor:exist:comp} are satisfied.
Let $\{\vec u_\comp, p_\comp\}_{0<\comp<1}$ denote the family of the weak solutions
to the problem \eqref{eq:comp}, \eqref{eq:ocomp}.
Let $\{\uinc, \pinc\}$ denote a weak solution to
the problem \eqref{eq:inc}, \eqref{eq:oinc}.

\subsection{Convergence of the Velocity}
\begin{thm}\label{th:weaklim}
If $\comp \to 0$, then
\begin{gather}
\vec u_\comp \wto \uinc \text{\; in }
L^2(0,T; \vecsp H_0^1(D)^d), \label{eq:th:weaklim:velconv} \\
\vec u_\comp \swto \uinc \text{\; in }
L^\infty(0,T; \vecsp L^2(D)^d), \label{eq:th:weaklim:velconv2} \\
\nabla p_\comp \swto \nabla\pinc \text{\; in }
H^{-1}(Q_T)^d, \label{eq:th:weaklim:presconv}
\end{gather}
where $\{\uinc, \pinc\}$ is the solution to \eqref{eq:inc}, \eqref{eq:oinc}
with
\begin{equation}\label{eq:th:weaklim:initial}
\uinco = P_J \vec u_0.
\end{equation}
\end{thm}
\begin{proof}
Consider the equality \eqref{eq:def:weak1:2} with arbitrary $\vec v \in V(\Omega)$:
\begin{multline}\label{eq:th:weaklim:1}
-\int_0^T \rinc (\vec u_\comp, \vec v) \, \varphi_t \, dt - \rinc (\vec u_0, \vec v) \varphi(0)
 = -\mu \int_0^T ((\vec u_\comp, \vec v)) \, \varphi \, dt + {}\\
{} + \int_0^T (\comp p_\comp \vec f, \vec v) \, \varphi \, dt +
\int_0^T \db{\rinc \vec f, \vec v} \varphi \, dt,
\end{multline}
where $\varphi\in C_0^\infty([0;T))$. Clearly $(\vec u_0, \vec v) =
(P_J\vec u_0, \vec v)$.
The estimate \eqref{eq:est1} implies that
$(\vec u_\comp)$ is bounded in $L^2(0,T;H_0^1(D)^d)$ and
$(\sqrt{\comp} p_\comp)$ is bounded in $L^\infty(0,T;L^2(D))$. Then, by
Alaoglu--Bourbaki theorem, for any sequence $\comp_n \to 0$, $n\to \infty$,
there exist $\boldsymbol w \in L^2(0,T;H_0^1(D)^d)$ and
a subsequence $(n_k)_{k\in\mathbb N}$ such that
$$
\vec u_{\comp_{n_k}} \wto \boldsymbol w \text{\; in } L^2(0,T;H_0^1(D)^d)
$$
when $k\to \infty$. Passing to the limit in \eqref{eq:th:weaklim:1}, we obtain
\begin{multline}\label{eq:th:weaklim:2}
-\int_0^T \rinc (\boldsymbol w, \vec v) \, \varphi_t \, dt
- \rinc (P_J\vec u_0, \vec v) \varphi(0) = {} \\
{} = -\mu \int_0^T ((\boldsymbol w, \vec v)) \, \varphi \, dt + 
0 + \int_0^T \db{\rinc \vec f, \vec v} \varphi \, dt,
\end{multline}
which also holds (by continuity) for each $\varphi \in H_0^1(0,T)$.
Passing to the limit in \eqref{eq:def:weak1:1}
assuming that $r\in \widehat{L}^2(D)$ and $\varphi \in C_0^\infty(0,T)$
we get
$$
\int_0^T ( \div \boldsymbol w, r ) \, \varphi \, dt = 0.
$$
Since $\div \boldsymbol w \in L^2(0,T;\widehat{L}^2(D))$,
this yields that $\boldsymbol w \in L^2(0,T;V(D))$.

Let us collect all the terms of \eqref{eq:th:weaklim:2} on its
left-hand side and denote the resulting expression by $\Phi(\vec v, \varphi)$.
For any fixed $\vec v\in H_0^1(D)^d$ 
$$\Phi(\vec v, \cdot) \in H^{-1}(0,T).$$
Integrating \eqref{eq:th:weaklim:2} by parts
we see that $\Phi(\vec v, \varphi) = \int_0^T \phi(\vec v, t) \varphi_t \,dt$,
where
$$
\phi(\vec v, t) = -\rinc (\boldsymbol w, \vec v) - 
\mu \int_0^t((\boldsymbol w, \vec v))\,  dt
+ \int_0^t \db{\rinc \vec f, \vec v} dt
$$
Clearly $\phi(\cdot, t) \in C(0,T;H^{-1}(D)^d)$.
From \eqref{eq:th:weaklim:2} and Lemma \ref{lem:nabla-inv}
it follows that there exists a
unique $Q\in C(0,T;\widehat{L}^2(D))$
such that for $t\in [0,T]$
$$
\phi(\cdot, t) = \db{\nabla Q(t),\cdot}
$$

Let $p$ denote the distributional derivative $\partial_t Q$ in 
$\mathcal D'(Q_T)$. Then $\{\boldsymbol w, p\}$ is a
weak solution to the problem \eqref{eq:inc}, \eqref{eq:oinc}
with the initial condition \eqref{eq:th:weaklim:initial}.
By Proposition \ref{th:exist:inc} this solution
is unique, so $\{\boldsymbol w, p\}=\{\uinc, \pinc\}$.
Hence the whole family $\vec u_\comp$
converges to $\uinc$ when $\comp\to 0$,
i.e. \eqref{eq:th:weaklim:velconv} holds.
(Otherwise this family would have a subsequence
which wouldn't converge to $\uinc$ weakly. But
such a subsequence would be bounded and applying
Alaoglu--Bourbaki theorem
and the arguments above we would come to a conclusion
that $\vec u_\comp$ converges to $\uinc$, which is a contradiction.)

As a consequence of estimate \eqref{eq:est1} and Alaoglu--Bourbaki theorem,
we can pass, if necessary, to another subsequence of $(n_k)$, so that
$\vec u_{\comp_{n_k}} \swto \vec w$ in $L^\infty(0,T;L^2(D)^d)$
for some $\vec w$.
Hence
$\forall \varphi \in C_0^\infty (0,T)$ 
and $\forall \vec r\in L^2(D)^d$
$$
\int_0^T (\vec u_{\comp_{n_k}}, \vec r) \varphi dt \to 
\int_0^T (\vec w, \vec r) \varphi dt, \quad k\to \infty.
$$
From this equation and \eqref{eq:th:weaklim:velconv}
we conclude that $\vec w = \uinc$ for a.e. $t\in[0,T]$.
Thus $\vec u_{\comp_{n_k}} \swto \uinc \text{\; in }
 L^\infty(0,T; \vecsp L^2(D)^d)$.
Since the whole family $\vec u_\comp$ converges to $\uinc$
when $\comp\to 0$, reasoning as above we conclude
that 
\eqref{eq:th:weaklim:velconv2} holds.

Finally, \eqref{eq:th:weaklim:presconv} follows from
\eqref{eq:th:weaklim:velconv} and
the passage to the limit in \eqref{eq:def:weak1:2}
when $\comp\to 0$.
\end{proof}
Theorem \ref{th:weaklim} is in good agreement with the results proved in
\cite{LionsMasmoudi,FeireislNovotnyNS}. But
the topology of the convergence \eqref{eq:th:weaklim:velconv},
\eqref{eq:th:weaklim:presconv}
can be strengthened using
the approach coming from the artificial compressibility method \cite{Temam}:

\begin{thm}\label{th:stronglim}
Suppose that $\uinco = P_J \vec u_0$. When $\comp \to 0$,
\begin{gather}
\vec u_\comp \to \uinc \text{\; in }
L^2(0,T; \vecsp H_0^1(D)^d), \label{eq:th:stronglim:velconv} \\
\vec u_\comp \to \uinc \text{\; in }
L^\infty(0,T; \vecsp L^2(D)^d), \label{eq:th:stronglim:velconv2} \\
\nabla p_\comp \to \nabla\pinc \text{\; in }
H^{-1}(Q_T)^d, \label{eq:th:stronglim:presconv}
\end{gather}
if and only if $\vec u_0\in J(D)$.
\end{thm}
\begin{proof} 
By Remark \ref{rem:embeddingSLC}
and Theorem \ref{th:weaklim} $\forall \vec v \in V(D)$
\begin{multline*}
\rinc (\overline{\vec u}_\comp(t) - \vec u_0, \vec v) + 0 = {}\\
{}=
-\mu \int_0^t ((\vec u_\comp, \vec v)) dt
- \eta \int_0^t (\div \vec u_\comp)(\div \vec v)dt
+ \int_0^t (\rinc + \comp p) (\vec f, \vec v) dt
\to {}\\
{} \to
-\mu \int_0^t ((\uinc, \vec v)) dt
- 0
+ \int_0^t \rinc (\vec f, \vec v) dt =
\rinc (\overline{\uinc}(t) - \uinco, \vec v).
\end{multline*}
By estimate \eqref{eq:est1} $|\overline{\vec u}_\comp(t)|$ is
bounded for $0<\comp<1$.
Then, since $V(D)$ is dense in $J(D)$,
by Banach--Steinhaus theorem $\forall \vec v \in J(D)$
\begin{equation}\label{eq:th:stronglim:aux}
(\overline{\vec u}_\comp(t), \vec v) \to (\overline{\uinc}(t), \vec v).
\end{equation}

Now consider a family of real numbers
\begin{multline*}
X_\comp = \rinc |\overline{\vec u}_\comp(T) - \overline{\uinc}(T)|^2 +
  \frac{\comp}{\rinc} |\overline{p}_\comp(T)|^2 + {}\\
{} +  2 \mu \int_0^T \|\vec u_\comp - \uinc\|^2 \, dt +
  2 \eta \int_0^T |\div (\vec u_\comp - \uinc)|^2 \, dt.
\end{multline*}
Expanding the parentheses we write
$$
X_\comp = X^{(1)} + X^{(2)}_\comp + X^{(3)}_\comp,
$$
where
\begin{gather*}
X^{(1)} = \rinc |\overline{\uinc} (T)|^2 + 2 \mu \int_0^T \|\uinc \|^2 \, dt,\\
X^{(2)}_\comp = -2\rinc (\overline{\vec u}_\comp(T), \overline{\uinc} (T))
 - 4 \mu \int_0^T ((\vec u_\comp, \uinc )) \, dt, \\
X^{(3)}_\comp = \rinc |\overline{\vec u}_\comp(T)|^2 +
  \frac{\comp}{\rinc} |\overline{p}_\comp(T)|^2
   + 2 \mu \int_0^T \|\vec u_\comp \|^2 \,dt+
   2 \eta \int_0^T |\div \vec u_\comp|^2 \, dt.
\end{gather*}
By energy equality \eqref{eq:energy-inc}
$$
X^{(1)} = \rinc|\uinc_0|^2 + \int_0^T (\rinc \vec f, \uinc)\, dt.
$$
From \eqref{eq:th:stronglim:aux} and
\eqref{eq:th:weaklim:velconv}
$$
X^{(2)}_\comp \to -2\rinc (\overline{\uinc}(T), \overline{\uinc}(T))
 - 4 \mu \int_0^T ((\uinc, \uinc )) \, dt = -2 X^{(1)}.
$$
Now we pass to the limit in energy equality \eqref{eq:energy} using
estimate \eqref{eq:est1}:
\begin{multline*}
X^{(3)}_\comp = \rinc|\vec u_0|^2 + \frac{\comp}{\rinc} |p_0|^2 +
 \int_0^T ((\rinc + \comp p) \vec f, \vec u_\comp)\, dt \to \\
 \to \rinc|\vec u_0|^2 + 0 +
  \int_0^T (\rinc \vec f, \uinc)\, dt = X^{(1)}
  + \rinc(|\vec u_0|^2-|\uinco|^2).
\end{multline*}
Hence we have proved that $X_\comp \to \rinc(|\vec u_0|^2-|\uinco|^2)$
when $\comp \to 0$.
Then \eqref{eq:th:stronglim:velconv} and \eqref{eq:th:stronglim:velconv2}
hold if and only if
$|\vec u_0| = |\uinco| = |P_J \vec u_0|$, i.e. $\vec u_0\in J(D)$.
And the passage to the limit in \eqref{eq:def:weak1:2} implies
\eqref{eq:th:stronglim:presconv}.
\end{proof}

\subsection{Convergence of the Pressure}

The important difference between the equations of compressible
fluid and the equations of incompressible fluid is that only the
former enjoy the mass conservation property.
Namely, let 
$$M(t) = \int_D \rho \, dx = \int_D (\rinc+\comp p) \, dx.$$
Then \eqref{eq:comp} implies
$$
\frac{dM}{dt} = 0.
$$
Consequently $\int_D p(t) \, dx$ doesn't depend on $t$.
In contrast, this is not true for the incompressible fluid,
since $\pinc$ can be shifted by arbitrary function of $t$.

However, for arbitrary $A\in \mathbb R$
the pressure $\pinc$ can be shifted so that
$\int_D \pinc(t) \, dx = A$ for a.e. $t\in[0,T]$.
(Let $C(t) = (\int \pinc(t) \, dx - A) / \int_D \,dx$, then the
substitution
$\pinc(t) := \pinc(t) - C(t)$ yields the desired result.)

\begin{thm}\label{th:weakplim}
Suppose that $\uinco=\vec u_0 \in J(D)$ and
$$\uinc \in \widetilde{W}^{1,2}(0,T;H_0^1(D)^d), \qquad
\pinc \in W^{1,2}(0,T;L^2(D)),$$
\begin{equation}\label{eq:massconsv}
\int_D q(t) \, dx = \int_D p_0 \, dx 
\quad\text{for a.e. } t\in [0,T]
\end{equation}
Then 
\begin{equation*}
p_\comp \swto \pinc \text{\; in } L^\infty(0,T;L^2(D)), \quad \comp\to 0.
\end{equation*}
\end{thm}
\begin{proof}
Assumptions of the theorem imply that the solution $\{\uinc, \pinc\}$
to the ``incompressible'' problem \eqref{eq:inc}, \eqref{eq:oinc}
belongs to the same functional space as the solution to the
``compressible'' problem \eqref{eq:comp}, \eqref{eq:ocomp}.
Then the difference of the solutions
$\{\vec U_\comp, P_\comp\}
\equiv\{\vec u_\comp - \uinc, p_\comp - \pinc\}$
is a weak solution to the problem \eqref{eq:gcomp}, \eqref{eq:ogcomp}
with initial conditions
$$
\vec U_\comp|_{t=0} = 0, \qquad P_\comp|_{t=0} = P_0\equiv p_0 - \pinc(0)
$$
and with non-homogeneous terms given by
$$
\sigma = - \comp\pinc_t, \quad \vec s = \comp \pinc \vec f.
$$
Assumption \eqref{eq:massconsv} implies 
$\pinc_t \in L^2(0,T;\widehat{L}^2(D))$.
Then \eqref{eq:est1} implies
\begin{gather*}
\|\vec U_\comp\|_{L^2(0,T;H_0^1(D)^d)}+
\|\vec U_\comp\|_{L^\infty(0,T;L^2(D)^d)} 
 \le C \sqrt{\comp} E, \\
 \|P_\comp\|_{L^\infty(0,T;\widehat{L}^2(D))} \le C E, \\
 E= \|P_0\|_{\widehat{L}^2(D)} + 
  \|\pinc_t\|_{L^2(0,T;L^2(D))} + 
  \sqrt{\comp}\|\vec f\|_\infty \|q\|_{L^2(0,T;L^2(D))}
\end{gather*}
thus we have proved the convergence of the velocity (once again).

Uniform boundedness of the pressure $P_\comp$, by Alaoglu--Bourbaki theorem,
implies that for every sequence $\comp_n\to 0$, $n\to \infty$,
there exists a subsequence $(n_k)_{k\in \mathbb N}$ such that
$$
P_{\comp_{n_k}} \swto P_* \quad \text{in } L^\infty(0,T;L^2(D)),
\quad k\to \infty,
$$
where $P_*\in L^\infty(0,T;\widehat{L}^2(D))$.
Passing to the limit in the equation \eqref{eq:def:weak1:2}
(as in the proof of Theorem \ref{th:weaklim})
we show that $\nabla P_*$ is zero. Then $P_*=0$
and the whole family $P_\comp$ converges to zero.
\end{proof}
\begin{rem}
The estimate \eqref{eq:est1} not only yields the \emph{fact}
of convergence, but also gives the \emph{rate} of the convergence.
In other words, \eqref{eq:est1} allows to estimate the
\emph{error of approximation} of a compressible fluid
by the incompressible one.
\end{rem}

It would be desirable to have \emph{strong} convergence of the pressure.
But in general case such convergence is impossible, since
(due to Proposition \ref{prp:embeddingSC})
$$\|p_\comp - \pinc\|_{L^\infty(0,T;L^2(D))}
 \ge \|p_0 - \overline{\pinc}(0)\|_{L^2(D)}.$$

\begin{thm}
Suppose the assumptions of Theorem~\ref{th:weakplim} are satisfied
and the domain $D$ is star-shaped with respect to some ball.
If, in addition,
$$
\pinc \in W^{2,2}(0,T;L^2(D))
\qquad
\text{and}
\qquad
p_0 = \overline{\pinc}(0),
$$
then
$$
\begin{gathered}
p_{\comp} \to \pinc
\quad\text{in}\quad
L^\infty(0,T;L^2(D)),
\qquad \comp \to 0.
\end{gathered}
$$
\end{thm}
\begin{proof}
Let $\vec w = -\mathcal{B} (\pinc_t)/\rinc$. By Proposition \ref{prp:bogovskii}
$$\vec w \in W^{1,2}(0,T;H_0^1(D)^d)$$
and
$$
\|\vec w\|_{W^{1,2}(0,T;H_0^1(D)^d)} \le C \|\pinc_t\|_{W^{1,2}(0,T;L^2(D))}.
$$
Then
$\{\vec U_\comp, P_\comp\}\equiv\{\vec u_\comp - \uinc -
 \comp \vec w, p_\comp - \pinc\}$
is a weak solution to the problem \eqref{eq:gcomp}, \eqref{eq:ogcomp}
with initial conditions
$$
\vec U_\comp|_{t=0} = \vec U_0 \equiv -\comp \vec \overline{w}(0),
 \qquad
 P_\comp|_{t=0} = P_0\equiv p_0 - \overline{\pinc}(0)=0
$$
and with non-homogeneous terms given by
$$
\sigma = 0, \quad \vec s = \comp(-\rinc\vec w_t + \mu \Delta \vec w +
\eta \nabla \div \vec w) + \comp \pinc \vec f.
$$

Clearly
$$
\|\vec s\|_{L^2(0,T;H^{-1}(D)^d)} \le
 \comp
 C C_1,
$$
where
$$
C_1=\|\vec w\|_{W^{1,2}(0,T;H_0^1(D)^d)} +
 \|\vec f\|_\infty \|\pinc\|_{L^2(0,T;L^2(D))}.
$$
Assumption \eqref{eq:massconsv} implies 
$\pinc_t \in L^2(0,T;\widehat{L}^2(D))$.
Then, from \eqref{eq:est1}
\begin{gather*}
\|\vec U_\comp\|_{L^2(0,T;H_0^1(D)^d)}
 \le C \sqrt{\comp} E, \\
 \|P_\comp\|_{L^\infty(0,T;\widehat{L}^2(D))} \le C E,
\end{gather*}
where
\begin{multline*}
E \equiv (\|\vec U_0\|_{L^2(D)^d} + \sqrt{\comp}\|P_0\|_{L^2(D)} + {}\\
{}+\frac{1}{\sqrt{\comp}}\|\sigma\|_{L^2(0,T;L^2(D))} +
 \|\vec s\|_{L^2(0,T;H^{-1}(D)^d)})/\sqrt{\comp} \le C C_1 \sqrt{\comp},
\end{multline*}
which yields the desired convergence when $\comp\to 0$.
\end{proof}

\subsection*{Acknowledgment}

The author is very grateful to V.Zh. Sakbaev, I.V. Volovich and E.G. Shifrin
for fruitful discussions of the paper and would also like to 
express gratitude to A. Shirikyan for warm hospitality during the
summer school on hydrodynamics at Cergy--Pontoise.



\end{document}